\documentclass[12pt]{amsart}
\usepackage{amsmath}
\usepackage{amsxtra}
\usepackage{amstext}
\usepackage{amssymb}
\usepackage{amsthm}
 \usepackage{amsfonts}
\usepackage{latexsym}
\usepackage{url}
\usepackage{dsfont} 
\usepackage{verbatim}
\usepackage{tabls}
\usepackage{graphicx}
\usepackage{rotating}
\usepackage{caption}
\usepackage[shortlabels]{enumitem}

\usepackage[pagebackref,hypertexnames=false, colorlinks, citecolor=red, linkcolor=red]{hyperref}
\usepackage[backrefs]{amsrefs}

\usepackage[inline,nomargin]{fixme}

\theoremstyle{plain}

\newtheorem{thm}{Theorem}[section]

\newtheorem{cor}[thm]{Corollary}
\newtheorem{lem}[thm]{Lemma}
\newtheorem{prop}[thm]{Proposition}

\newtheorem*{thmnn}{The Paneah-Logvinenko-Sereda Theorem}

\theoremstyle{definition}
\newtheorem{defn}[thm]{Definition}

\newtheorem{rem}[thm]{Remark}
\newtheorem*{rema}{Remark}
\newtheorem{cla}[thm]{Claim}

\numberwithin{equation}{section}
\def\supp{\operatorname{supp}}

\def\eps{\varepsilon}
\def\kap{\varkappa}

\def\doub{\operatorname{double}}
\def\wh{\widehat}
\def\wt{\widetilde}

\def\kap{\varkappa}
\def\M{\mathcal{M}}
\def\n{\mathfrak{n}}
\def\T{\mathbb{T}}
\def\H{\mathcal{H}}
\def\N{\mathbb{N}}
\def\R{\mathbb{R}}
\def\PLS{\operatorname{PLS}}
\def\D{\mathfrak{D}}

\def\XXint#1#2#3{{\setbox0=\hbox{$#1{#2#3}{\int}$}
     \vcenter{\hbox{$#2#3$}}\kern-.5\wd0}}

\begin{document}

\title[Fast decaying or sparsely supported Fourier transform]{Quantitative uniqueness properties for $L^2$ functions with fast decaying, or sparsely supported, Fourier transform}
\author{Benjamin Jaye}
\email{bjaye3@gatech.edu}
\address{School of Mathematical Sciences, Georgia Tech}
\author{Mishko Mitkovski}
\email{mmitkov@clemson.edu}
\address{School of Mathematical and Statistical Sciences, Clemson University}
\thanks{Research supported in part by NSF grants DMS-1800015 and DMS-1830128 to B.J.,  and DMS-1600874 to M.M.}

\begin{abstract}This paper builds upon two key principles behind the Bourgain-Dyatlov quantitative uniqueness theorem for functions with Fourier transform supported in an Ahlfors regular set.  We first provide a characterization of when a quantitative uniqueness theorem holds for functions with very quickly decaying Fourier transform, thereby providing an extension of the classical Paneah-Logvinenko-Sereda theorem.  Secondly, we derive a transference result which converts a quantitative uniqueness theorem for functions with fast decaying Fourier transform to one for functions with Fourier transform supported on a fractal set.  As well as recovering the result of Bourgain-Dyatlov, we obtain analogous uniqueness results for denser fractals.\end{abstract}

\maketitle

\section{Introduction}

The Fourier transform is the extension to $L^2(\R^d)$ of the operator which acts on $f\in L^1(\R^d)\cap L^2(\R^d)$ by $\wh{f}(\xi) = \int_{\R^d}f(t)e^{-2\pi i \xi \cdot t}dm_d(t)$, where $m_d$ is the $d$-dimensional Lebesgue measure. This paper builds upon two principles underlying Bourgain and Dyatlov's breakthrough uniqueness theorem for functions with Fourier transform supported in an Ahlfors regular set \cite{BD} (see Section \ref{BDapplication} below):
\begin{enumerate}
\item Classical uniqueness theorems for functions with compactly supported Fourier transform extend to functions with sufficiently fast decaying Fourier transform, and
\item these results can be transferred to uniqueness theorems for functions with sparsely supported Fourier transform by appealing to the Beurling-Malliavin theorem.
\end{enumerate}

In \cite{BD} these two principles are somewhat intertwined in the proof.  Our goal here is to separate them and develop some theory for a general weight function (in the spirit of Koosis' books \cite{Koo, Koo2}). By doing so, we \begin{enumerate}
\item obtain a characterization of when a uniqueness theorem holds for functions with fast decaying Fourier transform (under a convexity assumption on the weight), see Theorem \ref{PLSthm}, and
\item prove a general transference principle which converts a quantitative uniqueness theorem for functions with fast decaying Fourier transform to one for functions with sparsely supported Fourier transform (Theorem \ref{sparsesupport}).
      \end{enumerate}
As well as recovering the uniqueness result in \cite{BD}, this point of view enables one to obtain analogous results for functions whose Fourier transform is integrable with respect to the end-point weight given by $\exp\bigl(\frac{|t|}{\log(e+|t|)}\bigl)$.


\subsection{On the uniqueness (or strong annihilation) property for functions with fast decaying Fourier transform}

Denote by $m_d$ the Lebesgue measure on $\R^d$, $d\geq 1$.

\begin{defn} A Borel set $E\subset \R^d$ is $(\gamma,\ell)$-\emph{relatively dense} if $m_d(E\cap Q)\geq \gamma$
for any cube $Q\subset \R^d$ of side-length $\ell$.\end{defn}

The role of relatively dense sets in uniqueness theorems is exhibited by the classical Paneah-Logvinenko-Sereda theorem for band limited functions (\cite{Pan2, LS}, see also ~\cite{Pan1, Kov, HJ, Sch}), one of the prototypical forms of the uncertainty principle, see Chapter 1 of \cite{Pol}.

\begin{thmnn}  Fix $E\subset \R^d$.  For every $N>0$, there is a constant $C>0$ such that
\begin{equation}\label{recover}\|f\|_{L^2(\R^d)}\leq C\|f\|_{L^2(E)} \text{ for every }f\text{ with }\supp(\wh{f}\,)\subset B(0,N)
\end{equation}
if and only if $E$ is $(\gamma,\ell)$-relatively dense for some $\gamma\in (0,1)$ and $\ell>0$.\end{thmnn}

In particular, the theorem says that a band limited function can be reconstructed uniquely by its values on a relatively dense set.  The first result of this paper will be an extension of the Paneah-Logvinenko-Sereda theorem to functions which, instead of being band limited, have sufficiently fast decaying Fourier transform.

\begin{defn} A weight $W:[0,\infty)\to [0,\infty]$ has the \emph{Paneah-Logvinenko-Sereda (PLS) property} if, for every $d\in \mathbb{N}$, $\gamma\in (0,1)$, $\ell>0$, and $C_W>0$, there exists a finite constant $C=C(d,W,C_W,\gamma,\ell)>0$ such that if $f\in L^2(\R^d)$ satisfies \begin{equation}\label{CWbd}\int_{\R^d}|\wh{f}(\xi)W(|\xi|)|^2dm_d(\xi)\leq C_W^2\|f\|^2_{L^2(\R^d)},
\end{equation}
and $E$ is a $(\gamma, \ell)$-relatively dense set, then
\begin{equation}\label{plsineq}\|f\|_{L^2(\R^d)}\leq C\|f\|_{L^2(E)}.
\end{equation}
\end{defn}

Notice that the `if' direction of the Paneah-Logvinenko-Sereda theorem can be rephrased as the statement that for any $N>0$, the weight $$W(t) = \begin{cases} 1 \text{ for }|t|\leq N,\\ +\infty\text{ for }|t|>N \end{cases}$$
has the PLS property.

\begin{thm}\label{PLSthm}  Suppose that $W:[0,\infty)\to [0,\infty]$ satisfies
\begin{enumerate}
\item  $W(0)=1$, $W$ is non-decreasing, $W$ is lower-semicontinuous, and $\lim_{t\to\infty}W(t)=\infty$.
\item  the mapping $\log r \mapsto \log W(r)$ is convex\footnote{This permits an interval $(t_0,\infty)$ on which $W(t)=+\infty$.} on $[1,\infty)$.
\end{enumerate}
Then $W$ satisfies the PLS property if and only if \begin{equation}\label{logsum}\int_0^{\infty}\frac{\log W(t)}{1+t^2} dm_1(t)=\infty. \end{equation}
\end{thm}


Our motivation for formulating Theorem \ref{PLSthm} came from the paper \cite{BD}, where the following uniqueness theorem for functions with fast Fourier decay is presented:  If $\delta\in (0,1)$, and $\Theta(\xi) = \frac{|\xi|}{\log^{\delta}(e+ |\xi|)}$, then
$$\|e^{\Theta(\xi)}\wh{f}(\xi)\|_{L^2(\R)}\leq C_0\|f\|_{L^2(\R)} \text{ implies }c\|f\|_{L^2(\R)}\leq \|f\|_{L^2(E)}
$$
where $E$ is an infinite union of well separated intervals of some fixed side-length, and $c$ depends on $\delta, C_0$, and the sidelength of the intervals (see (1.6) in \cite{BD}).  Although stated in terms of intervals, it appears that one could adapt their proof to yield the stronger PLS property. Very recently, Han and Schlag \cite{HS} extended this estimate to several dimensions using Cartan set techniques.  

It is important to note that Theorem \ref{PLSthm} applies to the end-point weight $W(\xi)=e^{\Theta(\xi)}$ with $\Theta(\xi) = \frac{|\xi|}{\log(e+ |\xi|)}$.  It is remarked in \cite{BD} (see Remark 2 after Lemma 3.1 in \cite{BD}) that this weight does not grow quickly enough for their proof to be applicable.  Our approach is therefore necessarily rather different to that taken in \cite{BD}, or \cite{HS}, relying on quasianalyticity rather than harmonic measure estimates.

The constant $C>0$ in (\ref{plsineq}) that is obtained in the proof takes quite an explicit form that can be calculated given a particular choice of $W$ satisfying (\ref{logsum}), see Proposition \ref{uniqueness} below.  As an example, and since it will be used in the sequel, here we formulate a quantitative result for the end-point weight in the case $d=1$.

\begin{prop}\label{endpoint} Put $W(\xi)=e^{\Theta(\xi)}$ with $\Theta(\xi) = \frac{|\xi|}{\log(e+ |\xi|)}$.  Fix $\gamma\in (0,1)$, $\alpha\in (0,1]$ and $C_W>1$.  If $f\in L^2(\R)$ satisfies (\ref{CWbd}) with $W$ replaced by $W^{\alpha}$, and $E$ is $(\gamma,1)$-relatively dense, then
$$\|f\|_{L^2(\R)}\leq \Bigl(\frac{A}{\gamma}\Bigl)^{(\log [A\cdot  C_W])^{e^{A/\alpha}}}\|f\|_{L^2(E)}
$$
for an absolute constant $A>0$.
\end{prop}


\subsection{A Uniqueness result for functions whose Fourier transform is supported in a regular set}\label{BDapplication} In Section \ref{transfersection}, we will introduce a general condition of sparsity of the support of the Fourier transform of a function, which is a quantification of the classical \emph{short intervals condition} of Beurling (see Definition \ref{shortcover}).   In Theorem \ref{sparsesupport} we will prove quantiative uniqueness results for functions whose Fourier supported in such sparse sets.  

For simplicity, and for the sake of comparison with the results in \cite{BD, JZ}, in the introduction we restrict our attention to a special case.

\begin{defn}[$\varphi$-regular sets]\label{regdef}  Fix an increasing continuous function $\varphi:[0, \infty)\to [0, \infty)$ with $\varphi(0)=0$.  A set $Q$ is $\varphi$-regular if, for every $ N>1$, $t\in \R$, and $1\leq \ell\leq N$, there is a cover of the set $Q\cap [t-N,t+N]$ by $\varphi(N/\ell)$ intervals of length $\ell$. \end{defn}

For $\delta\in (0,1)$, every $\delta$-regular set in the terminology of \cite{BD} is $\varphi$-regular with $\varphi(t) = Ct^{\delta}$ (see Lemma 2.8 of \cite{BD}).

A uniqueness result for functions whose spectrum lies in a $\varphi$-regular set necessarily requires a bit more structure of the set $E$ beyond relative density (see point (2) in the comments following Theorem \ref{bourdyat}). For $\sigma>0$ set $E_{\sigma}$ to be the open $\sigma$-neighborhood of $E$.

\begin{thm}\label{bourdyat} Suppose that $\varphi$ satisfies $$\sum_{n\in \mathbb{N}}\frac{1}{n^2}\varphi(n)\leq C_\varphi\text{ for a constant }C_{\varphi}>1.$$    For every $\gamma,\sigma \in (0,1)$, there is a constant $C=C(C_\varphi, \sigma, \gamma)$ such that for any $(\gamma,1)$-relatively dense set $E$,  and $\varphi$-regular set $Q$,
$$\|f\|_{L^2(\R)}\leq C\|f\|_{L^2(E_{\sigma})} \text{ for every }f\in L^2(\R)\text{ with }\supp(\wh{f})\subset Q.
$$
Moreover, the constant $C(C_0, \sigma,\gamma)$ may be taken to be of the form
\begin{equation}\label{effectivebd}\exp\Bigl(\log(A/\gamma)\Bigl\{\exp\Bigl[\exp\Bigl(\frac{A\cdot C_{\varphi}}{\sigma}\Bigl)\Bigl]\Bigl\}\Bigl)
\end{equation}
 for an absolute constant $A>0$.
\end{thm}

Several comments regarding this result are in order:

(1)  Notice that in Theorem \ref{bourdyat} one controls the $L^2$ norm of $f$ by the $L^2$ norm on \emph{a $\sigma$-neighbourhood} of $E$.  It is natural to ask if this is necessary, and in the appendix we show that one cannot take $\sigma=0$ via a modification of a well-known example involving Riesz products. 

(2)  The uniqueness result in \cite{BD} corresponds to the case of a $\delta$-regular set with $\delta\in (0,1)$.  By characterizing the PLS property, we are able to obtain uniqueness results for fractal sets denser than the class considered by Bourgain and Dyatlov in~\cite{BD}.  The result in \cite{BD} is stated in terms of separated intervals, which is equivalent to a neighbourhood of a relatively dense set.

(3) In order to obtain the effective bound (\ref{effectivebd}), we incorporate the modification to the Bourgain-Dyatlov scheme introduced in Jin-Zhang \cite{JZ} -- replacing the use of the Beurling-Malliavan theorem (used in \cite{BD}) with the simpler effective multiplier theorem proved in \cite{JZ}.   In the setting of $\delta$-regular sets, Jin and Zhang \cite{JZ} obtained a version of Corollary \ref{bourdyat} with an effective bound of the form $\exp\exp\exp\bigl(\frac{C(\sigma, \gamma)}{(1-\delta)}\log \tfrac{1}{1-\delta}\bigl)$  (see Theorem 4.4 in \cite{JZ}).   For $\delta$-regular sets, the constant $C_{\varphi}$ may be taken to be $C/(1-\delta)$ for some absolute $C>0$ and so we obtain a bound of the form $\exp\exp\exp(\frac{C(\sigma, \gamma)}{1-\delta})$.   

\section{Background material in quasianalytic functions required for Theorem \ref{PLSthm}}  

The main direction of the proof of Theorem \ref{PLSthm} is the proof that (\ref{logsum}) implies that the PLS property holds.  The idea behind the proof is simple:  The property on $W$ yields that $f$ belongs to a certain quasi-analytic class.  Using the localization principle behind the proof of the Paneah-Logvinenko-Sereda theorem~\cite{LS} as presented in \cite{Kov} or \cite{Sch}, we can reduce matters to a Remez-type inequality for quasianalytic functions, which is provided by an extension to several variables of a theorem of Nazarov-Sodin-Volberg \cite{NSV}.  This is carried out in Section \ref{suf}.  For readers who are not so concerned about the particular form of the constant $C>0$ in (\ref{plsineq}) that arises in the proof, we also provide a short proof of a more qualitative statement relying only on the Denjoy-Carleman theorem.\\

On the other hand, if (\ref{logsum}) fails to hold, the Paley-Wiener multiplier theorem, see \cite{Koo} p.97,  yields the existence of functions supported on arbitrarily small balls for which $\int_{\R^d}|\wh{f}(\xi)|^2W(|\xi|)^2d\xi<\infty$, thereby exhibiting that such $W$ fail to satisfy the PLS property.  For the benefit of the reader we sketch the argument in Section \ref{nec}.\\

Until the conclusion of the proof of Theorem \ref{PLSthm}, assume that $W$ is a weight satisfying hypotheses (1) and (2) of Theorem \ref{PLSthm}.  There is no loss of generality by assuming that $W\equiv 1$ on $[0,1]$, and $W$ grows faster than any power function at infinity,\footnote{The Fourier transform of a compactly supported bump function is in the Schwartz class, so power bounded weights certainly fail to satisfy the PLS property.} so we shall always do so. Set
$$M_n = \sup_{\xi\in \R}\frac{|\xi|^n}{W(|\xi|)} = \sup_{t\geq 1}\frac{t^n}{W(t)}.
$$
Notice that $M_n$ is an increasing log-convex sequence: $M_n^2\leq M_{n-1}M_{n+1}$.  Therefore, setting $M_0=\max_{\xi\in \R}\frac{1}{W(|\xi|)}= 1$ we have that the sequence $\mu_n = \frac{M_{n-1}}{M_n}$ is non-increasing, and $\mu_n\leq 1$. \\

We begin by revisiting some very well-known elementary inequalities (e.g. \cite{Koo}) in order to make our discussion self-contained.  The property (2) of the weight $W$ is used in the following lemma:

\begin{lem}\label{convlem}For $r>1$ with $W(r)<\infty$, there exists an integer $n\geq 0$ with
$$\log W(r)\leq (n+1)\log r- \log M_n.
$$
\end{lem}

\begin{proof} Fix $r>1$ and choose some supporting line to the graph $\{(\log t, \log W(t))\,:\, t>0\}$ with finite slope $\nu$ at the point $(\log r, \log W(r))$ ($\nu\geq 0$ since $W$ is increasing).   With $n$ equal the integer part of $\nu$ we observe that $\log M_n\leq (n+1)\log r - \log W(r)$ (see \cite{Koo} p.99-100), as required.
\end{proof}


\begin{prop}\label{elementary}  The following inequalities hold:
$$\sum_{n}\mu_n \leq \int_1^{\infty}\frac{\log W(t)}{t^2}dm(t) \leq \sum_{n}\mu_n+1.
$$
\end{prop}


To prove this consider the Ostrowski function
$$\rho(r) = \sup_{n\in \mathbb{N}}\frac{r^n}{M_n}.
$$
Notice that, since the sequence $\mu_n$ is decreasing, $\rho(r) = \Pi_{\{n:\,r\mu_n>1\}}(r\mu_n).$ The proposition is an immediate consequence of combining the following two lemmas.

\begin{lem}\label{koosis} For $r>1$, \begin{equation}\label{koosisstatement}\log \rho(r)\leq \log W(r)\leq \log \rho(r)+\log r.\end{equation}
\end{lem}

\begin{proof} The left hand inequality is trivial.  If $r$ lies in the set $I = \{W<\infty\}$, then we use Lemma \ref{convlem} to fix $n\geq 0$ with $\log W(r)\leq (n+1)\log r- \log M_n$.  But then $\log W(r)\leq (n+1)\log r-\log M_n\leq \rho(r)+\log r$, and so (\ref{koosisstatement}) holds in $I$.  In the case that $I=[0, r_0)$ is a bounded interval, and $W(r_0)=\infty$, then $\lim_{r\to r_0^-}W(r)=\infty$ ($W$ is increasing and lower semi-continuous), and since (\ref{koosisstatement}) holds in $I$, we get that $ \rho(r_0)=+\infty$ and hence $ \rho(r) = \infty$ for all $r>r_0$ ($\rho$ is  non-decreasing).  Finally, if $I=[0,r_0]$ and $W(r_0)<\infty$, $M_n \leq r_0^n$ for every $n$, and therefore $\rho(r) \geq \sup_{n\in \mathbb{N}} \bigl(\frac{r}{r_0}\bigl)^n=\infty$ for $r>r_0$.  
\end{proof}

\begin{lem}\label{katznelson} The following identity holds:$$\int_1^{\infty}\frac{\log \rho(t)}{t^2}dm(t) = \sum_{n}\mu_n $$
\end{lem}

\begin{proof} The left hand side equals (using that $\mu_n \leq 1$)
$$\int_1^{\infty}\frac{\sum_{n:\, t\mu_n>1}\log(t\mu_n)}{t^2} dm(t) = \sum_{n}\int^{\infty}_{1/\mu_n}\frac{\log(\mu_n t)}{t^2}dm(t).
$$
With a change of variable, we see that
$$\int^{\infty}_{1/\mu_n}\frac{\log(\mu_n t)}{t^2}dm(t) = \mu_n \int^{\infty}_{1}\frac{\log(t)}{t^2}dm(t)=\mu_n,
$$
as required.
\end{proof}

\subsection{The Nazarov-Sodin-Volberg Theorem}  Given any logarithmically convex sequence $\M= \{M_n\}_{n\in \mathbb{N}}$ with $M_0=1$, we consider the class $\mathcal{C}_{\M}([0,1])$ of smooth functions which satisfy $\|f^{(n)}\|_{L^{\infty}[0,1]}\leq M_n$ for every $n\geq 0$.   A sequence $\M$ \emph{generates a quasi-analytic class} if whenever $f\in \mathcal{C}_{\M}([0,1])$ vanishes to infinite order at a point in $[0,1]$ ($f^{(k)}(x_0)=0$ for every $k\geq 0$ for some $x_0\in [0,1]$), then $f\equiv 0$ on $[0,1]$.  The Denjoy-Carleman theorem (see e.g. \cite{Koo}) ensures that $\M$ generates a quasi-analytic class if and only if \begin{equation}\label{quasi}\sum_{n=1}^{\infty}\frac{M_{n-1}}{M_n}=\infty.\end{equation}
With a slight abuse of notation, we call a logarithmically convex sequence $\M$ satisfying (\ref{quasi}) \emph{quasi-analytic}.

For $f\in \mathcal{C}_{\mathcal{M}}([0,1])$,  the Bang degree $\mathfrak{n}_f$ is defined by
\begin{equation}\begin{split}\mathfrak{n}_f &= \sup\Bigl\{N: \sum_{\log \|f\|_{L^{\infty}([0,1])}^{-1}< n\leq N}\frac{M_{n-1}}{M_{n}}<e\Bigl\}. \end{split}\end{equation}
A powerful theorem of Bang (see \cite{Ba} or \cite{NSV}) states that the Bang degree controls the number of zeros of a function $f\in \mathcal{C}_{\mathcal{M}}([0,1])$ counting multiplicities.  It is therefore natural that it should depend on both the growth of the ratios of $M_{n-1}/M_n$ and a lower bound for $\|f\|_{L^{\infty}([0,1])}$. For our purposes we will want uniform bounds on the Bang degree of a function given the class $\M$.  Therefore, we set, for $t\in (0,1]$,
$$\n_{\M,t} = \sup\Bigl\{N: \sum_{- \log t< n\leq N}\frac{M_{n-1}}{M_{n}}<e\Bigl\},
$$
so if $f\in C_{\M}([0,1])$ satisfies $\|f\|_{L^{\infty}([0,1])}\geq t$, then $\n_f\leq \n_{\M,t}$.  Following \cite{NSV}, we also define (compare with (1.7) in \cite{NSV})
$$\gamma_{\M}(n)=\sup_{1\leq j\leq n}j\Bigl[\frac{M_{j+1}M_{j-1}}{M_{j}^2}-1\Bigl],\text{ and }\Gamma_{\M}(n) = 4e^{4+4\gamma_{\M}(n)}.
$$
We are now in a position to state the Nazarov-Sodin-Volberg theorem, which builds upon the techniques developed by Bang \cite{Ba}.
\begin{thm}[Theorem B from \cite{NSV}]\label{NSVthm1}  Suppose that $f\in \mathcal{C}_{\M}([0,1])$.  Then for any interval $I\subset [0,1]$ and measurable set $E\subset I$ with $m_1(E)>0$, we have
$$\sup_{ I}|f|\leq \Bigl(\frac{\Gamma_{\M}(2\mathfrak{n}_f)|I|}{m(E)}\Bigl)^{2\mathfrak{n}_f}\sup_{ E}|f|.
$$
\end{thm}

Again, the constant in this inequality must depend on the ratio of the value of $t= \|f\|_{L^{\infty}([0,1])} $ to its apriori upper bound of $M_0=1$: the smaller the value of $t$, the more zeroes $f$ can have in the interval $[0,1]$ while controlling the size of a fixed number of derivatives.

Theorem \ref{NSVthm1} does not require the sequence $\M$ to be quasi-analytic, but we shall only use it in this case.  

Since there has been interest in obtaining quantitative uniqueness bounds, see e.g. \cite{JZ}, we thought it worthwhile to present Theorem \ref{NSVthm1}, where the constant is rather sharp\footnote{We also like its proof.}.   However, if the reader is not bothered by the particular form of the constant in Theorem \ref{NSVthm1}, then the following qualitative result can be quickly derived from the Denjoy-Carleman theorem.

\begin{rem}[A quick qualitative bound]\label{softrem}  If $\gamma>0$, $t>0$, and $\M$ is a quasi-analytic sequence, then there is a finite constant $C=C(\gamma, t, \M)$ such that whenever $f\in C_{\M}([0,1])$ satisfies $\|f\|_{L^{\infty}[0,1]}\geq t$ and $E\subset [0,1]$ satisfies $m_1(E)\geq \gamma$, then
\begin{equation}\label{soft}\|f\|_{L^{\infty}([0,1])}\leq C(\gamma, t, \M)\|f\|_{L^{\infty}(E)}.
\end{equation}\end{rem}
\begin{proof}[Proof of Remark \ref{softrem}] Suppose the result fails to hold, then for some $\gamma>0 $ and $t>0$, there is a sequence $\{f_n\}_n\in C_{\M}([0,1])$ satisfying $\|f_n\|_{L^{\infty}([0,1])}\geq t$ and a set $E_n\subset [0,1]$ with $m_1(E_n)\geq \gamma$ such that $\|f_n\|_{L^{\infty}(E)}\leq \frac{1}{n}\|f_n\|_{L^{\infty}([0,1])}\leq \frac{1}{n}$.  For any $k\geq 0$, the sequence $\{D^k f_n\}_n$ is certainly equicontinuous, and so, with the aid of a diagonal argument and relabelling the sequence if necessary, we may assume that $f_n$ converges uniformly to a function $f\in C_{\M}([0,1])$.  But then $\|f\|_{L^{\infty}([0,1])}\geq t$ (since $[0,1]$ is compact), while $f\equiv 0$ on the set $E = \bigcap_n \bigcup_{m\geq n} E_m$ (if $x\in E$, then  $x\in E_{n_m}$ for some subsequence $n_m\to \infty$, but then $|f(x)| = \lim_{m\to \infty}|f_{n_m}(x)|=0$).  Of course, $m_1(E)\geq \gamma$.  However, a smooth function that vanishes on a set of positive measure has a zero of infinite order (for instance, at each Lebesgue point of the zero set), so $f\equiv 0$ on $[0,1]$ since $\M$ generates a quasi-analytic class.  This contradiction  establishes (\ref{soft}).
\end{proof}

We shall require an extension of Theorem \ref{NSVthm1} for quasi-analytic functions of several variables.  To do this we shall appeal to an inductive argument of Fontes-Merz \cite{FM}.

For $Q\subset \R^d$ a cube (whose sides are parallel to the coordinate axes), we say that $f:Q\to \R$ lies in $C_{\mathcal{M}}(Q)$ if for any multi-index $\alpha$ with order $|\alpha|:=\alpha_1+\dots+\alpha_d=n$, it holds that $\|D^{\alpha}f\|_{L^{\infty}(Q)}\leq M_n$.

\begin{prop}\label{NSVrd}   For any $d\geq 1$, $t\in (0,1]$, quasianalytic class $\M$, and $s\in (0,1]$, there is a finite constant $\Theta_{\M}(d,t, s)$ such that for any cube $Q\subset \R^d$ of sidelength $1$ and whenever $f\in C_{\M}(Q)$ satisfies $\|f\|_{L^{\infty}(Q)}\geq t$, and $E\subset Q$ is a Borel measurable set with $m_d(E)\geq s$,
$$\|f\|_{L^{\infty}(Q)}\leq \Theta_{\M}(d,t, s)\|f\|_{L^{\infty}(E)}.
$$
Moreover, we have the estimate
$$\Theta_{\M}(d,t, s)\leq \Theta_{\M}\Bigl(1,t, \frac{s}{2}\Bigl)\Theta_{\M}\Bigl(d-1, \frac{t}{\Theta_{\M}(1,t, \tfrac{s}{2})}, \frac{s}{2}\Bigl).
$$
\end{prop}

Observe that Theorem \ref{NSVthm1} ensures that
\begin{equation}\label{theta1bd}\Theta_{\M}(1,t,s) \leq \Bigl(\frac{\Gamma_{\M}(2\mathfrak{n}_{\M,t})}{s}\Bigl)^{2\mathfrak{n}_{\M,t}},
\end{equation}
so one can calculate an effective bound on $\Theta_{\M}(d,\cdot,\cdot)$ for any dimension, albeit of a tower exponential form.  

\begin{proof}We follow the inductive scheme in \cite{FM}.  Without loss of generality, assume $Q=[0,1]^d$.  The base case $d=1$ is covered by the Nazarov-Sodin-Volberg theorem (or Remark \ref{softrem}).  Suppose now that $d>1$ and the proposition is proved for $d-1$.  Fix $f\in C_{\M}([0,1]^d)$, $\|f\|_{L^{\infty}([0,1]^d)}\geq t$ and $E\subset [0,1]^d$ with $m_d(E)\geq s$.  For $x\in \R^d$, put $x=(x',u)$ where $x'\in \R^{d-1}$ and $u\in \R$.  We set $E_{u} = \{x'\in \R^{d-1}: (x',u)\in E\}$.  Define the set
$$L = \{u\in [0,1]: m_{d-1}(E_{u})\geq \frac{1}{2}m_d(E)\}.
$$
Then $$m_d(E) \leq \int_{L}m_{d-1}(E_u) dm_1(u)+ \int_{[0,1]\backslash L}m_{d-1}(E_u)dm_1(u).$$  We bound the first integral by $m_1(L)$ (since $E_u\subset [0,1]^{d-1}$), and the second integral by $\frac{1}{2}m_d(E)$ (for $u\in [0,1]\backslash L$, $m_{d-1}(E_u)< \frac{1}{2}m_d(E)$).  Therefore, $m_1(L)\geq \frac{1}{2}m_d(E)\geq \frac{s}{2}$.

 Suppose $(x',u)\in [0,1]^d$ satisfies $|f(x',u)|=\sup_{x\in [0,1]^d}|f(x)|$.  Applying the $d=1$ case to the function $f(x',\,\cdot\,)\in C_{\M}([0,1])$ and the set $L$ yields that
$$t \leq \|f\|_{L^{\infty}([0,1]^d)}=|f(x',u)|\leq \Theta_{\M}(1,t,s/2)\sup_{u\in L}|f(x',u)|.
$$
Let $\eps>0$ and fix $u_0\in L$ with $|f(x', u_0)|+\eps\geq \sup_{u\in L}|f(x',u)|$.  Then by definition of $L$, $m_{d-1}(E_{u_0})\geq m_d(E)/2\geq s/2$.  Also,
$$\sup_{y'\in [0,1]^{d-1}}|f(y',u_0)|\geq |f(x', u_0)|\geq \frac{t}{\Theta_{\M}(1,t,s/2)}-\eps.
$$
Consequently, we may apply the inductive hypothesis that the proposition holds for $d-1$  to the function $f(\cdot, u_0)$ and the set $E_{u_0}$ to obtain
$$\sup_{y'\in [0,1]^{d-1}}|f(y', u_0)|\leq \Theta_{\M}\Bigl(d-1,  \frac{t}{\Theta_{\M}(1,t,\tfrac{s}{2})}-\eps, \frac{s}{2}\Bigl)\sup_{y'\in E_{u_0}}|f(y',u_0)|.
$$
But if $y'\in E_{u_0}$, then $(y',u_0)\in E$ so $\sup_{y'\in E_{u_0}}|f(y',u_0)|\leq \sup_{x\in E}|f(x)|$.  Letting $\eps\to 0$, we conclude that
$$\|f\|_{L^{\infty}([0,1]^d)}\leq \Theta_{\M}(1,t,s/2)\Theta_{\M}\Bigl(d-1,  \frac{t}{\Theta_{\M}(1,t,\tfrac{s}{2})}, \frac{s}{2}\Bigl)\cdot \sup_{x\in E}|f(x)|,
$$
as required.\end{proof}

We will require an $L^2$-version of Proposition \ref{NSVrd}.

\begin{cor}\label{NSVthm} Suppose that $f\in \mathcal{C}_{\M}([0,1]^d)$ satisfies $\|f\|_{L^{\infty}([0,1]^d)}\geq t>0$.  Then for any Borel measurable set $E\subset [0,1]^d$ with positive measure, we have
$$\int_{[0,1]^d} |f|^2dm_d\leq \frac{2\Theta_{\M}(d,t,m_d(E)/2)^2}{m_d(E)}\int_E|f|^2 dm_d.
$$
\end{cor}

\begin{proof} Consider the set $\wt{E} = \Bigl\{x\in E: |f(x)|^2\leq \frac{2}{m_d(E)}\int_E|f|^2dm_d\Bigl\}$.  Then $m_d(\wt{E})\geq \frac{1}{2}m_d(E)$. Applying Proposition \ref{NSVrd} with the set $\wt{E}$, it follows that
\begin{equation}\nonumber\sup_{[0,1]^d}|f|\leq \Theta_{\M}(d,t,m_d(E)/2)\sup_{\wt{E}}|f|.
\end{equation}
But $\sup_{\wt{E}}|f|^2\leq \frac{2}{m_d(E)}\int_E|f|^2dm_d$, and so
$$\int_{[0,1]^d}|f|^2 dm_d\leq \sup_{[0,1]^d}|f|^2\leq \frac{2\Theta_{\M}(d,t,m_d(E)/2)^2}{m_d(E)}\int_{E}|f|^2dm_d,$$
as required.
\end{proof}

\section{The sufficiency of (\ref{logsum}) for the PLS property}\label{suf}  

Without loss of generality, we shall put $\ell=1$ in the definition of relative density (for any $\ell>0$, $W$ satisfies (\ref{logsum}) if and only if $W(\ell\, \cdot\,)$ does).  Suppose that
$$\int_0^{\infty}\frac{\log W(t)}{1+t^2} dm_1(t)=\infty.
$$
With $M_n = \max_{\xi\in \R}\frac{|\xi|^n}{W(|\xi|)}$ and $\mu_n = \frac{M_{n-1}}{M_n}$, we infer from Proposition \ref{elementary} that $\sum_n \mu_n=+\infty$, so $\M=\{M_n\}_{n\geq 0}$ is a quasi-analytic class with $M_0=1$.

A slightly modified quasi-analytic class will arise naturally in the proof, so we introduce it here.  For $A>1$, we define
\begin{equation}\label{MA}\M_{A} = \{\wt{M}_n\}_{n\geq 0}\text{ with }\wt{M}_0=1\text{ and }\wt{M}_n = A^n\frac{M_{n+d}}{M_d}.\end{equation}
Observe that $\M_{A}$ is a log-convex sequence since $\M$ is log-convex.

\begin{prop}\label{uniqueness} There exists $A=A(d)>1$ such that for any $f\in L^2(\R^d)$ satisfying (\ref{CWbd}), $\gamma\in (0,1)$, and $(\gamma,1)$-relatively dense subset $E\subset \R^d$,
$$\int_{\R^d}|f|^2 dm_d\leq \frac{4}{\gamma} \Theta_{\M_A}\Bigl(d,\frac{1}{C_W A^{1+d}M_d}, \frac{\gamma}{2}\Bigl)^2\int_{E}|f|^2 dm_d.
$$
\end{prop}






\begin{proof} Suppose $\|f\|_{L^2(\R^d)}=1$.  Partition $\R^d$ into cubes of side-length $1$.  Fix $B>2$.  A cube $Q$ is said to be bad if there exists a multi-index $\alpha$ such that
\begin{equation}\label{badn}\int_Q|D^{\alpha}f|^2 dm_d >B^{2(|\alpha|+1)}M_{|\alpha|}^2C_W^2\int_Q|f|^2 dm_d.
\end{equation}
If a cube isn't bad, then it is called good.  If $Q$ is a good cube, then we have good derivative control:
\begin{equation}\label{goodcond}\int_Q|D^{\alpha}f|^2 dm_d \leq B^{2(|\alpha|+1)}M_{|\alpha|}^2C_W^2\int_Q|f|^2 dm_d \text{ for every }\alpha\in \mathbb{Z}_+^d.
\end{equation}
Notice that if $|\alpha|=n$, then by Plancherel's identity, we have that $D^{\alpha}f\in L^2(\R^d)$, and moreover
\begin{equation}\label{planch}\begin{split}\int_{\R^d}|D^{\alpha}f|^2 dm_d & = (2\pi)^{2n}\int_{\R^d}|\xi^{\alpha}|^{2}|\wh{f}(\xi)|^2 dm_d(\xi)\\
&\leq (2\pi)^{2n} \Bigl[\max_{\xi\in \R^d}\frac{|\xi|^n}{W(|\xi|)}\Bigl]^2\int_{\R^d}|\wh{f}(\xi)|^2W(|\xi|)^2dm_d(\xi)\\&\leq (2\pi)^{2n} M_n^2C_W^2.
\end{split}\end{equation}
Therefore, if $\mathcal{B}_n$ denotes the union of all cubes that are bad for derivatives of order $n$ (i.e. the union of intervals for which (\ref{badn}) holds for some multi-index of order $n$), then
$$\int_{\mathcal{B}_n}|f|^2 dm_d\leq \frac{1}{B^{2(n+1)}M_n^2C_W^2}\sum_{\alpha:|\alpha|=n}\int_{\R^d}|D^{\alpha} f|^2 dm_d \leq \frac{(2\pi)^{2n}C(n)}{B^{2(n+1)}},
$$
where $C(n)$ denotes the number of possible multi-indices of order $n$.  By induction one can readily see that $C(n)\leq (n+1)^d$.

Consequently, if $\mathcal{B}$ denotes the union of all bad cubes, and $B$ is large enough, then
$$\int_{\mathcal{B}}|f|^2 dm_d\leq \frac{1}{B^2}\sum_{n\geq 0}\frac{(2\pi)^{2n}(n+1)^d}{B^{2n}}\leq \frac{1}{2} = \frac{\|f\|_{L^2(\R^d)}^2}{2}.
$$
and so
\begin{equation}\label{goodcontribute}\int_{\bigcup\{ Q\text{ good}\}}|f|^2dm_d\geq \frac{1}{2}\|f\|_{L^2(\R^d)}^2.
\end{equation}

Now fix a good cube $Q$ (which we recall has sidelength $1$).  Recall the elementary Sobolev inequality (see Chapter 1 of \cite{Maz})
$$\|g\|_{L^{\infty}(Q)}\leq C(d)\|g\|_{L^2(Q)}+C(d)\sum_{|\alpha|=d}\|\partial^{\alpha} g\|_{L^2(Q)}.
$$
From (\ref{goodcond}) we infer that for every $n\geq 0$ and $|\alpha|=n$,
\begin{equation}\label{unider}\begin{split}\|D^{\alpha}f\|_{L^{\infty}(Q)}&\leq C(d)(B^{n+d+1}M_{n+d})C_W\|f\|_{L^2(Q)}\\& \leq A^{n+d+1}M_{n+d}C_W\|f\|_{L^2(Q)},\end{split}\end{equation}
for $A=A(d)$.

Consider the function $\wt{f} = \frac{f}{A^{d+1}M_{d}C_W\|f\|_{L^2(Q)}}$.  Then $\wt{f}$ belongs to the class $C_{\M_{A}}(Q)$ with the sequence $\M_{A}$ defined in (\ref{MA}).  Also
$$\|\wt{f}\|_{L^{\infty}(Q)}\geq \frac{1}{ C_W A^{1+d}M_d}.
$$
Therefore, applying Corollary \ref{NSVthm} with the function $\wt{f}$ and the set $E\cap Q$, which has measure at least $\gamma$, results in $$\int_{Q} |\wt{f}|^2dm_d\leq \frac{2}{\gamma} \Theta_{\M_A}\Bigl(d,\frac{1}{ C_W A^{1+d}M_d}, \frac{\gamma}{2}\Bigl)^2\int_{E\cap Q}|\wt{f}|^2 dm_d.
$$
By homogeneity, this inequality also holds with $f$ replacing $\wt{f}$.

Finally, summing over good cubes, we conclude from (\ref{goodcontribute})
$$\frac{1}{2}\int_{\R^d}|f|^2 dm_d \leq \frac{2}{\gamma} \Theta_{\M_A}\Bigl(d,\frac{1}{C_W A^{1+d}M_d}, \frac{\gamma}{2}\Bigl)^2\int_{E}|f|^2 dm_d.
$$
Proposition \ref{uniqueness} is proved.\end{proof}

\subsection{The proof of Proposition \ref{endpoint}}

We will make a rough calculation of the order of magnitude of the constant appearing in Proposition \ref{uniqueness} (see also (\ref{theta1bd})).  To this end, it is slightly more convenient to work with the weight 
$$W(t) = \begin{cases}
1 \text{ for }t\leq e\\
e^{\frac{t}{\log t} -e}\text{ for }t>e,
\end{cases}
$$
which changes the value of $C_W$ by at most an absolute constant multiple.  Adjusting the absolute constant appearing in the statement of Proposition \ref{endpoint} if necessary, we may assume that $\alpha\leq 1/100$, $1/\alpha\in \N$ and $C_W\geq 100$. Throughout the proof $C>0$ denotes an absolute constant that can change from line to line.

Setting $M_n = \sup_{t>0}\frac{t^n}{W(t)}$  and $M_{\alpha,n} = \sup_{t>0}\frac{t^n}{W^{\alpha}(t)}$, we observe that $M_{\alpha, n} = (M_{n/\alpha})^{\alpha}$. Put $\mathcal{M} := \{M_{\alpha, n}\}_n $ and
$$\M_{A} = \{\wt{M}_n\}_{n\geq 0}\text{ with }\wt{M}_0=1\text{ and }\wt{M}_n = A^n\frac{M_{\alpha,n+1}}{M_{\alpha,1}},$$
where $A>0$ is the absolute constant appearing in the statement of Proposition \ref{uniqueness}.

At least for $n\geq 1/\alpha$, the value $t_n$ for which $M_n$ is achieved is the solution to the equation 
\begin{equation}\label{diff}
\frac{n}{t_n}-\frac{1}{\log t_n}+\frac{1}{(\log t_n)^2}=0, \text{ or }n = \frac{t_n}{\log t_n}\Bigl(1-\frac{1}{\log t_n}\Bigl).
\end{equation}
From the intermediate value theorem we deduce that, for $n\geq 1/\alpha$,
\begin{equation}\label{tnbds}n\log n <t_n<2n \log n.\end{equation} Therefore, we have
$$\frac{M_{\alpha,n-1}}{M_{\alpha,n}} \geq \frac{t_{n/\alpha}^{n-1}/W^{\alpha}(t_{n/\alpha})}{t_{n/\alpha}^n/W^{\alpha}(t_{n/\alpha})} = \frac{1}{t_{n/\alpha}}\geq \frac{1}{2[n/\alpha]\cdot \log (n/\alpha)}.
$$
Consequently, 
$$\mathfrak{n}_{\mathcal{M}_{A},t} \leq \sup_N\Bigl\{\sum_{-\log t<n<N} \frac{1}{[(n+1)/\alpha]\cdot\log ((n+1)/\alpha)}\lesssim Ae\Bigl\},
$$
but 
\begin{equation}\begin{split}\nonumber \sum_{-\log t<n<N} &\frac{1}{[(n+1)/\alpha]\cdot\log ((n+1)/\alpha)}\\&\gtrsim  \alpha\Bigl(\log \log \bigl\{[N+1]/\alpha\bigl\} - \log\log \bigl\{(\log[e/t])/\alpha\bigl\}\Bigl)
\end{split}\end{equation}
and therefore
$$\mathfrak{n}_{\mathcal{M}_{A},t} \leq \Bigl(\frac{\log (e/t)}{\alpha}\Bigl)^{\exp(C/\alpha)}.$$
In our case, $t = \frac{1}{C_W A^2M_{\alpha, 1}}\gtrsim \frac{ \alpha}{C_W\cdot \log 1/\alpha}$, so,
$$\mathfrak{n}_{\mathcal{M}_{A},t} 
\leq (\log C_W)^{\exp(C/\alpha)}.
$$
Finally, in order to calculate a bound for $\Gamma_{\M_A}(2\mathfrak{n}_{\M_A,t})$, we need to estimate, for $j\in \mathbb{N}$,
$$\frac{\wt{M}_{j-1}\wt{M}_{j+1}}{\wt{M}_j^2}=\frac{M_{j,\alpha}M_{j+2,\alpha}}{M_{j+1, \alpha}^2}.
$$
By definition,
$$M_{j+1, \alpha}\geq \max\Bigl\{\frac{t_{j/\alpha}^{j+1}}{W^{\alpha}(t_{j/\alpha})}, \frac{t_{(j+2)/\alpha}^{j+1}}{W^{\alpha}(t_{(j+2)/\alpha})}\Bigl\},
$$
and hence
\begin{equation}\label{Mratiocases}\frac{M_{j,\alpha}M_{j+2,\alpha}}{M_{j+1, \alpha}^2}\leq \frac{t_{(j+2)/\alpha}}{t_{j/\alpha}} 
\end{equation}
First, by (\ref{diff}), we observe that for $n\geq 1/\alpha$,
$$t_{n}-t_{n-1} = n\log t_n \Bigl(1-\frac{1}{\log t_n}\Bigl)^{-1}- (n-1)\log t_{n-1}\cdot \Bigl(1-\frac{1}{\log t_{n-1}}\Bigl)^{-1},
$$
but since $\alpha\leq 1/100$, 
$$1-\frac{1}{\log t_{n-1}}\gtrsim 1,
$$
and therefore, employing the mean value inequality, we obtain that
$$|t_{n}-t_{n-1}|\lesssim \log t_{n}+\frac{n}{t_{n-1}}|t_n-t_{n-1}|+\frac{(n-1)}{t_{n-1}\log t_{n-1}}|t_n-t_{n-1}|.
$$
Plugging in the bounds (\ref{tnbds}) and simplifying yields that
\begin{equation}\label{consecutivej}|t_n/t_{n-1}-1|\lesssim 1/n.
\end{equation}
Employing straightforward inequalities yields the following bound
\begin{equation}\begin{split}\nonumber\frac{t_{(j+2)/\alpha}}{t_{j/\alpha}}& = \exp \Bigl(\sum_{n=j/\alpha+1}^{(j+2)/\alpha}\log t_n/t_{n-1}\Bigl)\leq  \exp\Bigl(\sum_{\ell=j/\alpha+1}^{(j+2)/\alpha}|t_n/t_{n-1}-1|\Bigl).
\end{split}\end{equation}
But (\ref{consecutivej}) ensures that
$$\sum_{n=j/\alpha+1}^{(j+2)/\alpha}|t_n/t_{n-1}-1|\lesssim \frac{1}{\alpha}\cdot\frac{\alpha}{j}\lesssim 1/j,
$$
and therefore, recalling (\ref{Mratiocases}),
$$j\Bigl(\frac{\wt{M}_{j-1}\wt{M}_{j+1}}{\wt{M}_j^2} -1\Bigl)\lesssim j\cdot \sum_{n=(j-1)/\alpha+1}^{(j+1)/\alpha}|t_n/t_{n-1}-1|\lesssim 1,
$$
Consequently, $\Gamma_{\M_A}(2\mathfrak{n}_{\M_A,t})\lesssim 1.$  (A bound of the form $\Gamma_{\M_A}(2\mathfrak{n}_{\M_A,t}) \lesssim e^{1/\alpha}$ would still be permissible.)  Proposition \ref{endpoint} now follows from Proposition \ref{uniqueness} (see (\ref{theta1bd})).

\section{The necessity of (\ref{logsum}) for the PLS property}\label{nec}

We only consider $d=1$.  We shall assume $$\int_0^{\infty}\frac{\log W(t)}{1+t^2}dt<\infty,$$
and therefore (Proposition \ref{elementary}), $\sum_{n}\mu_n<\infty$.

We shall sketch the Paley-Wiener construction (also the construction used in many presentations of the Denjoy-Carleman theorem, see e.g. \cite{Co, Koo}) to show that there exist functions $f$ supported on arbitrarily small intervals with $\int_{\R}|\wh{f}(\xi)|^2 W(|\xi|)^2d\xi<\infty$.  Therefore $W$ fails to have the PLS property. 

Fix $\eps>0$.  Choose $n_0\geq 10$ such that $\sum_{n\geq n_0}\mu_n<\eps$.  We set
$$\wh{f}(\xi) = M_{n_0-1}\Bigl(\frac{\sin((\eps/n_0) \xi)}{(\eps/n_0)\xi}\Bigl)^{2n_0}\prod_{n\geq n_0}\frac{\sin (\mu_n\xi)}{\mu_n\xi}.$$
As in (for example) Koosis, \cite{Koo}, p. 90-91, we infer that
\begin{itemize}
\item $\wh{f}$ is the Fourier transform of a function that vanishes outside of an interval of width $C\eps$, for some absolute constant $C>0$, and
\item for $n\geq 0$, and $|\xi|>1$,
\begin{equation}\label{momentineq}\begin{split}|\xi|^n|\wh{f}(\xi)|&\leq \max(M_n, M_{n_0}) \Bigl(\frac{n_0}{\eps}\Bigl)^{n_0}\Bigl(\frac{|\sin ((\eps/n_0)\xi)|}{(\eps/n_0)|\xi|}\Bigl)^{n_0+1}\\&\leq C(n_0, \eps)\frac{\max(M_n,M_{n_0})}{|\xi|^{n_0+1}}.
\end{split}\end{equation}
\end{itemize}
From Lemma \ref{convlem} we therefore infer that for $|\xi|>1$ there exists $n$ such that
$$\log W(|\xi|)\leq (n+1)\log|\xi|-\log M_n.
$$
But when combined with (\ref{momentineq}) this yields that
$$W(|\xi|)\leq \frac{C(n_0,M_{n_0}, \eps)}{|\wh{f}(\xi)||\xi|^{n_0}}.
$$
Therefore
$$\int_{\R}|\wh{f}(\xi)|^2W(|\xi|)^2 dm(\xi)<\infty.
$$

\section{From fast decay to sparse support:  A transference principle}\label{transfersection}  To develop a transference principle we shall lean on the scheme developed in \cite{BD}.  In particular our considerations are based on use of a simple variant of the Buerling-Mallivan multiplier theorem (see, e.g. \cite{JZ, HMN, Koo2}), which will restrict our discussion to uniqueness theorems in one dimension.  Han and Schlag \cite{HS} adapted the techniques in \cite{BD, JZ} to derive a multi-dimensional analogue of the Bourgain-Dyatlov fractal uncertainty principle for certain Ahlfors regular subsets of $\R^d$ with (possibly distorted) product structure, still making use of a multiplier theorem in one dimension.  There are analogues of Theorem \ref{sparsesupport} below in this product setting.

The condition of sparsity that arises is a modification of the short intervals condition (cf. the Beurling gap theorem \cite{Pol}) taking into account that
\begin{itemize}
\item[--]  the result here is an $L^2$-theorem, so the condition of sparsity should be stable under translations in the Fourier domain, and,
\item[--]  our conclusion is quantitative, so there should be some uniformity in the shortness condition.
\end{itemize}

With this in mind, we make the following definitions.

\begin{defn}\label{shortcover} Fix a weight $W:[0,\infty)\to [0, \infty)$ with $ W> 1$ on $[1,\infty)$.
\begin{itemize}
\item A collection $\{\mathcal{J}_n\}_n$ is a $W$-short cover of a set $Q\subset \R$ if for every $n\in \mathbb{N}$, $\mathcal{J}_n$ is comprised of intervals of length $\Omega_n = \log W(e^{n})$ such that
     \begin{enumerate}
     \item $\displaystyle \bigcup_{J\in \mathcal{J}_n}J\supset Q\cap ([-e^{n+1},-e^n]\cup [e^n, e^{n+1}]),$ and
     \item $\displaystyle\|\{\mathcal{J}_n\}_n\|_W:=\sum_{n\in \mathbb{N}}\Bigl(\frac{\Omega_n}{e^{n}}\Bigl)^2\text{card}(\mathcal{J}_n)<\infty.$
    \end{enumerate}
\item A set $Q$ is called $W$-sparse if, for every $t\in \R$, the set $Q-t$ has a $W$-short cover $\{\mathcal{J}^{(t)}_n\}_n$, and moreover
$$\|Q\|_W = \sup_{t\in \R}\;\inf\limits_{\substack{\mathcal{J}^{(t)}_n \text{ a }W-\text{short}\\ \text{cover of }Q-t}}\|\{\mathcal{J}_n^{(t)}\}_n\|_W<\infty.
$$
\end{itemize}
\end{defn}

\begin{rema} If $Q$ has a short $W$-cover and $\wt W\leq W$, then $Q$ has a short $\wt W$ cover.  To see this cover each interval $J\in \mathcal{J}_n$ of length $\log W(e^n)$ with no more than $\lfloor\log W(e^n)/\log \wt W(e^n)\rfloor+1$ intervals of length $\log \wt W(e^n)$, and set $\wt{\mathcal{J}}_n$ to be the resulting collection of intervals of length $\log \wt W(e^n)$.  Thus
$$\sum_n\Bigl(\frac{\log \wt W(e^n)}{e^n}\Bigl)^2\text{card}(\wt{\mathcal{J}}_n) \leq 2 \sum_n\Bigl(\frac{\log W(e^n)\log \wt W(e^n)}{e^{2n}}\Bigl)\text{card}(\mathcal{J}_n),
$$
and the right hand side is smaller than $2\|\{\mathcal{J}_n\}_n\|_W$.  As such, a slower growing weight $W$ will have more $W$-sparse sets associated to it.
\end{rema}

The transference principle that we prove will (necessarily) be for neighborhoods of relatively dense sets.



\begin{thm}\label{sparsesupport}  Fix an increasing weight $W\geq 1$ such that
\begin{enumerate}
\item for every $\alpha>0$, $W^{\alpha}$ has the PLS property\footnote{Notice that conditions (1)--(2) of Theorem \ref{PLSthm}, and the validity of (\ref{logsum}), are invariant under this transformation.},
\item there is a constant $C_{\text{doub}}$ such that $\log W(et)\leq C_{\text{doub}}\log W(t)$, and $\log W(t)\leq t/4$, for every $t>1$
\end{enumerate} For every $\Lambda>0$, $\gamma>0$ and $\sigma>0$, there is a constant $C=C(C_{\text{doub}}, W, \Lambda, \gamma,\sigma)$, such that for every $W$-sparse set $Q$ with $\|Q\|_W\leq \Lambda$, every $(\gamma,1)$-relatively dense set $E$, and every $f\in L^2(\R)$ with $\supp(\wh{f})\subset Q$,
\begin{equation}\label{sparseunique}\|f\|_{L^2(\R)}\leq C\|f\|_{L^2(E_{\sigma})}.
\end{equation}
\end{thm}


\begin{rema}One cannot expect Theorem \ref{sparsesupport} (or Theorem \ref{bourdyat} below) to hold for in the case $\sigma=0$.  An example is included in the appendix.
\end{rema}

The proof of Theorem \ref{sparsesupport} consists of a reorganization of the ideas presented in \cite{BD} -- incorperating the Jin-Zhang \cite{JZ} effective multiplier theorem -- combined with a localization trick.  

\section{An application of the multiplier theorem}\label{BMsection}


We begin with the effective multiplier theorem proved by Jin-Zhang \cite{JZ} which replaces use of the Beurling-Malliavan theorem in scheme of \cite{BM}.  See \cite{Koo2} for (much) more information on the Beurling-Malliavin theorem and the instances when it can be applied.  
  The precise formulation we use may be found in Appendix B of Han-Schlag \cite{HS}.

For a (Borel measurable) function $f:\R\to \R$ satisfying
$$\int_{\R}\frac{|f(t)|}{1+t^2}dt<\infty,
$$
we define the Hilbert transform
$$\mathcal{H}(f)(x) = \frac{1}{\pi} P.V.\int_{\R}\Bigl(\frac{1}{x-t}+\frac{t}{1+t^2}\Bigl)f(t)dm(t).
$$

\begin{thm}[Appendix B of \cite{HS}]\label{BMthm}  Fix $\sigma>0$.  Suppose that $\wt W=e^{\wt \Omega}: \R\to \R$ satisfies $\wt W\geq 1$, $$\int_0^{\infty}\frac{\wt \Omega(t)}{1+t^2}dm(t)<\infty \text{ and }\|\mathcal{H}(\wt\Omega)'\|_{\infty}<\frac{\pi}{2}\sigma.$$   then  there exists $\varphi\in L^2(\R)$ satisfying
\begin{enumerate}
\item $\supp(\varphi)\subset [0,\sigma)$,
\item $|\wh{\varphi}(\xi)|  \gtrsim \sigma^{10}\wt W^{-1}(\xi)$ for $\xi\in [-3/4, 3/4]$, and
\item $|\wh{\varphi}(\xi)|\leq \wt W^{-1}(\xi)$ on $\R$.
\end{enumerate}
\end{thm}

 In addition to giving quantitative results, the use of Theorem \ref{BMthm} has the additional benefit that its proof is simpler than the proof of the full Buerling-Malliavan theorem (for instance, as it is presented in \cite{HMN}).


In this section we shall apply Theorem \ref{BMthm} to prove the following proposition.

\begin{prop}\label{phimult}  Suppose that $W=e^{\Omega}$ satisfies the assumptions of Theorem \ref{sparsesupport}.  There is a constant $c>0$, that may depend on $C_{\doub}$, such that for every $\sigma>0$ and $\Lambda>1$, the following statement holds:

If $Q\subset \R$ satisfies $\|Q\|_{W}\leq \Lambda$, then there exists $\varphi\in C^{\infty}_0(\R)$ satisfying
\begin{enumerate}
\item $\supp(\varphi)\subset [0,\sigma]$,
\item $\int_{[-1,1]}|\wh\varphi(\xi)|^2 d\xi\gtrsim \sigma^{20}$,
\item $|\wh{\varphi}(\xi)|\lesssim \exp\Bigl(-c\sigma \sqrt{|\xi|}\Bigl)$ on $\R$,  and
\item $|\wh{\varphi}(\xi)|\lesssim \exp\bigl(-\frac{c\sigma\Omega(\xi)}{\Lambda+1}\bigl)$ on $Q_2$ (the $2$-neighborhood of $Q$).
\end{enumerate}

\end{prop}

The remainder of this section is devoted to the proof of this proposition.   All constants may depend on $C_{\doub}$ without further mention.\\

Suppose that $\{\wt{\mathcal{J}}_n\}_n$ is a $W$-short cover of a set $Q$ with $\|\{\wt{\mathcal{J}}_n\}_n\|_W\leq \Lambda.$  We begin by regularizing the cover.

\begin{lem}\label{canonicalcover}Suppose that $\{\wt{\mathcal{J}}_n\}_n$ is a $W$-short cover of a set $Q$.  Then there is a $W$-short cover $\{\mathcal{J}_n\}_n$ of $Q$ satisfying that for every $n$, $\{\frac{1}{2}J:J\in \mathcal{J}_n\}$ are pairwise disjoint, and $\|\{\mathcal{J}_n\}_n\|_W\leq 7\|\{\wt{\mathcal{J}}_n\}_n\|_W$.
\end{lem}

\begin{proof}  Fix $n$, and pick a maximal collection of $\frac{\Omega_n}{2}$-separated points $\{t_m\}_m$ in $Q\cap [[-e^{n+1},-e^n]\cup[e^n,e^{n+1}]]$.  Consider the intervals $J_m$ centred at $t_m$ of sidelength $\Omega_n$.  Since $\{t_m\}_m$ are $\frac{\Omega_n}{2}$-separated, $\frac{1}{2}J_m$ are disjoint.  On the other hand, by maximality, $Q\cap [[-e^{n+1},-e^n]\cup[e^n,e^{n+1}]]\subset \bigcup_m J_m$.  But at most $7$ intervals $J_m$ can intersect any interval $J\in \wt{\mathcal{J}}_n$.  Therefore, if $\mathcal{J}_n = \{J_m\}_m$, then $\text{card}(\mathcal{J}_n)\leq 7\cdot \text{card}(\wt{\mathcal{J}}_n)$.
\end{proof}

Going back to our $W$-short cover $\{\wt{\mathcal{J}}_n\}_n$, we set $\{\mathcal{J}_n\}_n$ as in the lemma,  and so $\|\{\mathcal{J}_n\}_n\|_W\leq 7\Lambda$. Put $\mathcal{J} = \bigcup_n \mathcal{J}_n$.

\begin{cla}\label{cla1} For each $J\in \mathcal{J}_n$, $$3J\subset [-e^{n+2}, -e^{n-1}]\cap [e^{n-1}, e^{n+2}].$$

\end{cla}

\begin{proof}Since $\log W(e^n)\leq \frac{e^n}{4}$, the claim follows from the fact $J\in \mathcal{J}_n$ intersects $[e^n, e^{n+1}]\cap [-e^{n+1}, -e^n]$.
\end{proof}

\begin{cla}\label{cla2}  There is a constant $C>0$ depending on $C_{\text{doub}}$ such that any interval $3J$, $J\in \mathcal{J}$, can intersect at most $C$ of the intervals $\{3I\}_{I\in \mathcal{J}}$.
\end{cla}

\begin{proof}Fix $J\in \mathcal{J}$, so $J\in \mathcal{J}_n$ for some $n$.  From Claim \ref{cla1}, we infer that if $I\in \mathcal{J}$ satisfies $3I\cap 3J\neq \varnothing$, then $I\in \mathcal{J}_m$ with $|n-m|\leq 4$.  Fix such an $m$ and consider all $I\in \mathcal{J}_m$ with $3I\cap 3J\neq \varnothing$.  Since $W$ satisfies the doubling condition, $C_{\text{doub}}^{-4}\leq |\log W(e^m)/\log W(e^n)|=|\Omega_m/\Omega_n|\leq C_{\text{doub}}^4$.  Consequently, any such interval $I$ is contained in the $15C_{\text{doub}}^4$ dilation of $J$, and has length at least $C_{\text{doub}}^{-4}\ell(J)$.   Finally, since the collection of intervals $\{\frac{1}{2}I: I\in \mathcal{J}_m\}$ are pairwise disjoint, there can be at most $\frac{15C_{\text{doub}}^4}{\frac{1}{2}C_{\text{doub}}^{-4}} = 30C_{\text{doub}}^8$ such intervals $I\in \mathcal{J}_m$.  Since there are at most nine choices of $m$, the lemma is proved. \end{proof}

Now, observe that since $\log W$ is doubling, we obtain from Claim \ref{cla1} that
\begin{equation}\label{WdoubJ}\begin{split}\log W(t)\leq &\Omega_{n+2}\leq C_{\text{doub}}^2\log W(e^n)\leq C_{\text{doub}}^3\log W(t)\\&\text{ for any }t\in 3J\in \mathcal{J}_n.\end{split}\end{equation}
 Whence
\begin{equation}\begin{split}\label{Woncover}\int_{\{\bigcup 3J: J\in \mathcal{J}\}}\frac{\log W(t)}{1+t^2}dm(t)&\leq C\sum_n \Omega_n^2 \frac{\text{card}(\mathcal{J}_n)}{e^{2n}}\\&\leq C\|\{\mathcal{J}_n\}_n\|_W\leq C\Lambda.
\end{split}\end{equation}

   Fix $\eta\in C^{\infty}_0([-3,3])$ with $\eta\equiv 1$ on $[-2,2]$.  
   
   For every $J\in \mathcal{J}_n$,  set $\eta_J = \eta\bigl(\frac{\,\cdot-x_J}{\Omega_n}\bigl)$, where $x_J$ is the center of $J$. 
   Set $$\wt{\Omega}^{(1)}(t) =\sum_n\sum_{J\in \mathcal{J}_n} \Omega_{n+2}\eta_J.$$
Observe from (\ref{WdoubJ}) that \begin{equation}\label{Omega1largeoncover}\wt{\Omega}^{(1)}\geq \log W\text{ on }\bigcup_{J\in \mathcal{J}} 2J,\end{equation} 
while from (\ref{Woncover}) we derive that
\begin{equation}\label{omega1poisson}\int_0^{\infty}\frac{\wt{\Omega}^{(1)}(t)}{1+t^2} dm(t)\lesssim \Lambda+1.
\end{equation}
\begin{cla}
\begin{equation}\label{HilbertLip1}
\|\H(\wt{\Omega}^{(1)})'\|_{\infty}\lesssim 1+ \Lambda.
\end{equation}
  \end{cla}
   \begin{proof} Insofar as $\eta\in C^{\infty}_0([-3,3])$,
   $$|\H(\eta)'(x)|\lesssim \frac{1}{1+x^2}\text{ for every }x\in \R,$$
   and consequently
   \begin{equation}\label{HJbd}|\H(\eta_J)'(x)|\lesssim \frac{\Omega_n}{\Omega_n^2+(x-x_J)^2}\text{ for every }x\in \R.
   \end{equation}
Therefore, due to (\ref{HJbd}) and (\ref{WdoubJ})
$$\bigl|\H(\wt{\Omega}^{(1)})'(x)\bigl|\lesssim \sum_n\sum_{J\in \mathcal{J}_n}\frac{\Omega_n^2}{\Omega_n^2+(x-x_J)^2}$$
Recall that for each $n\in \N$, the points $x_J$, $J\in \mathcal{J}_n$, are $\Omega_n/2$ separated, and so
\begin{equation}\label{singlenHbd}\sum_{J\in \mathcal{J}_n}\frac{\Omega_n^2}{\Omega_n^2+(x-x_J)^2}\lesssim \sum_{k=0}^{\infty}\frac{1}{1+k^2}\lesssim 1.
\end{equation}
On the other hand, if $|n-\ln x|>2$, then $|x-x_J|\gtrsim e^n$ for every $J\in \mathcal{J}_n$.  Whence,
$$\sum_{n:\, |n-\ln x|>2}\sum_{J\in \mathcal{J}_n}\frac{\Omega_n^2}{\Omega_n^2+(x-x_J)^2}\lesssim \sum_n\Bigl(\frac{\Omega_n}{e^n}\Bigl)^2\text{card}(\mathcal{J}_n)\lesssim \Lambda.
$$
Since at most $5$ of the natural numbers $n$ can satisfy $|n-\ln x|\leq 2$, we conclude that (\ref{HilbertLip1}) holds from (\ref{singlenHbd}).\end{proof}

Now put $\wt{\Omega}^{(2)}(t) = \sqrt[4]{1+t^2}$.  It is a straightforward calculation to show that
\begin{equation}\label{HilbertLip2}\|\H(\wt{\Omega}^{(2)})'\|_{\infty}\lesssim 1.
\end{equation}

\begin{proof}[Proof of Proposition \ref{phimult}]As a consequence of (\ref{omega1poisson}), (\ref{HilbertLip1}), and (\ref{HilbertLip2}), we observe Theorem \ref{BMthm} is applicable with the weight
$$\wt{\Omega} = \frac{c\sigma}{\Lambda+1}\wt{\Omega}^{(1)}+c\sigma \wt{\Omega}^{(2)}
$$
 for a suitable small constant $c>0$, and provides us with a multiplier $\varphi$.  Properties (1)--(3) of Proposition \ref{phimult} follow immediately.  It is also immediate from  (\ref{Omega1largeoncover}) that $|\wh{\varphi}(\xi)|\lesssim \exp\bigl(-\frac{c\sigma\Omega(\xi)}{\Lambda}\bigl)$ on $\bigcup_n\bigcup_{J\in \mathcal{J}_n} 2J$.

 Finally, we observe that since the weight $W$ is increasing, there exists (a smallest) $n_0$ depending on $W$, such that $\ell(J)=\Omega_n\geq 4$ whenever $J\in \mathcal{J}_n$, $n \geq n_0$.  Setting $Q_2$ to be the closed $2$-neighbourhood of $Q$, we therefore infer that $$\bigcup_{n\geq n_0}\bigcup_{J\in \mathcal{J}_n} 2J\supset Q_2\cap [(-\infty, -e^{n_0}]\cup [e^{n_0},\infty)]$$
(recall that $\bigcup_{n\geq n_0}\bigcup_{J\in \mathcal{J}_n} J\supset Q\cap [(-\infty, -e^{n_0}]\cap [e^{n_0},\infty)]$).  But if $t<e^{n_0}$, we have that $\Omega(t)\leq 4$.  Taking into account that $\sigma<1$, we get that 
$|\wh{\varphi}(\xi)|\lesssim \exp\bigl(-\frac{c\sigma\Omega(\xi)}{\Lambda+1}\bigl)$ on $Q_2$.
 \end{proof}




\section{The proof of Theorem \ref{sparsesupport}}

We need a simple preparatory lemma.

\begin{lem}\label{BDlem} Suppose that $\supp(\wh{f})\subset \bigcup_k I_k$, where $I_k = [t_k-1,t_k+1]$.  Fix $\kap>0$, and a sequence $\varphi_k\in L^2(\R)$ with $\|\wh \varphi_k\|_{L^2([-1,1])}^2\geq \kap>0$.  Then
$$\|f\|_{L^2(\R)}^2\leq \frac{1}{\kap}\int_{-2}^2\sum_k \bigl\|\wh{f}(\cdot-\tau-t_k)\wh{\varphi_k}\bigl\|_{L^2(\R)}^2d\tau.
$$
\end{lem}

\begin{proof}  Fix $k$ and observe that, with a change of variable,
\begin{equation}\begin{split}\nonumber\|\wh{f}\|_{L^2(I_k)}^2 & = \|\wh{f}(\cdot -t_k)\|_{L^2([-1,1])}\leq \frac{1}{\kap}\int_{-1}^1\int_{-1}^1 |\wh{f}(\xi-t_k)|^2|\wh{\varphi_k}(\zeta)|^2 d\zeta d\xi\\
&\leq \frac{1}{\kap}\int_{-2}^2 \int_{-1}^1 |\wh{f}(\zeta-\tau-t_k)|^2|\wh{\varphi_k}(\zeta)|^2 d\zeta d\tau\\&\leq \frac{1}{\kap}\int_{-2}^2\bigl\|\wh{f}(\cdot-\tau-t_k)\wh{\varphi_k}\bigl\|_{L^2(\R)}^2d\tau.
\end{split}\end{equation}
The lemma follows by summation over $k$ (along with Plancherel's identity).
\end{proof}

Let us now begin the proof in earnest.  Recall that $Q_2$ is the closed $2$-neighbourhood of $Q$.

Fix $\{t_{\ell}\}_{\ell}$ to be a maximal one-separated subset\footnote{A maximal set satisfying $|t_{\ell}-t_{\ell'}|\geq 1$ if $\ell\neq \ell'$.} of $Q_2$, so $Q_2\subset \bigcup_{\ell}I_{\ell}$, where $I_{\ell} = [t_{\ell}-1, t_{\ell}+1]$.

Suppose that $Q$ and $f$ satisfy the hypotheses of Theorem \ref{sparsesupport} (so $\|Q\|_W\leq \Lambda$ and $\supp(\wh f)\subset Q$).  For every $\ell$, we can apply the construction of Proposition \ref{phimult} with $\sigma>0$ to obtain a function $\varphi_{\ell}$ satisfying
\begin{enumerate}
\item $|\wh\varphi_{\ell}| \lesssim W^{-2\alpha}$  on $Q_2-t_{\ell}$
\item $|\wh{\varphi_{\ell}}(t)|\leq e^{-c_0\sigma\sqrt{\max(1,|t|)}}$ on $\R$
\item $\supp(\varphi_{\ell})\subset [0, \sigma],$ and
\item $\|\wh\varphi_{\ell}\|_{L^2([-1,1])}\gtrsim \sigma^{10}.$
\end{enumerate}
where $\alpha = \frac{c_0\sigma}{\Lambda+1}$, and $c_0>0$ can depend on $C_{\doub}$.

We will need the following simple auxiliary lemma.

\begin{lem}\label{gensumbd}  For any $g\in L^2(\R)$,
$$\sum_{\ell}\int_{-2}^2\|\wh{g}(\cdot-\tau-t_{\ell}) e^{-c_0\sigma\sqrt{|\,\cdot\,|}}\|^2_{L^2(\R)}d\tau\lesssim \frac{1}{\sigma^2}\|g\|_{L^2(\R)}^2
$$
\end{lem}

\begin{proof}  The left hand side of the inequality is bounded by
$$\int_{-2}^2 \sum_{\ell}\|\wh{g} e^{-c_0\sigma\sqrt{|(\,\cdot\, +\tau+t_{\ell})|}}\|^2_{L^2(\R)}d\tau.
$$
But, since the points $\{t_{\ell}\}_{\ell}$ are one-separated, $$\sup_{\xi,\tau \in \R}\sum_{\ell} e^{-2c_0\sigma \sqrt{|(\xi +\tau+t_{\ell})|}}\lesssim \frac{1}{\sigma^2}.$$ and the lemma follows.
\end{proof}

For any $\tau\in [-2,2]$, consider function $$f_{\tau,\ell} = \mathcal{F}^{-1}(\wh{f}(\cdot-\tau-t_\ell)\wh{\varphi_{\ell}}) = (f\cdot e_{\tau+t_{\ell}})*\varphi_{\ell},$$
where $e_{t}(x) = e^{2\pi i tx}$.   The function $f_{\tau, \ell}$ has its Fourier transform supported in the set $Q_2-t_{\ell}$ and so satisfies that
$$|\wh{f_{\tau,\ell}}|\lesssim |\wh{f}(\cdot-\tau-t_{\ell})|\sqrt{|\wh\varphi_{\ell}|}W^{-\alpha} \text{ on }\R.
$$
Consequently,
$$\|\wh{f_{\tau, \ell}}W^{\alpha}\|_{L^2(\R)}\leq C\|\wh{f}(\cdot-\tau-t_{\ell})\sqrt{|\wh\varphi_{\ell}}|\|_{L^2(\R)}
$$
and so by combining property (2) of $\varphi_{\ell}$ and Lemma \ref{gensumbd} we infer that
\begin{equation}\label{sumcontrol}\begin{split}\int_{-2}^2\sum_{\ell} \bigl\|\wh{f_{\tau, \ell}}W^{\alpha}\bigl\|_{L^2(\R)}^2d\tau
\lesssim \frac{1}{\sigma^2}\|f\|_{L^2(\R)}^2.\end{split}
\end{equation}
We apply a localization technique:  Fix $\D>1$.  We call a pair $(\tau,\ell)$ \emph{bad} if
$$\bigl\|\wh{f_{\tau, \ell}}W^{\alpha}\bigl\|_{L^2(\R)}^2>\D\bigl\|\wh{f_{\tau, \ell}}\bigl\|_{L^2(\R)}^2.
$$
Otherwise $(\tau, \ell)$ is called \emph{good}.  If $(\tau,\ell)$ is good, then $f_{\tau, \ell}$ satisfies the condition to apply the PLS property with $C_W=\D$.   Notice first that,
\begin{equation}\begin{split}\nonumber \int_{-2}^2\sum_{\ell: \; (\tau, \ell)\text{ bad }} \bigl\|\wh{f_{\tau, \ell}}\bigl\|_{L^2(\R)}^2d\tau &\leq \frac{1}{\D} \int_{-2}^2\sum_{\ell} \bigl\|\wh{f_{\tau, \ell}}W^{\alpha}\bigl\|_{L^2(\R)}^2d\tau\\
&\stackrel{(\ref{sumcontrol})}{\lesssim} \frac{1}{\D\sigma^2}\|f\|_{L^2(\R)}^2. \end{split}\end{equation}
Employing Lemma \ref{BDlem} with $\kap = c\sigma^{20}$  with a suitable constant $c>0$ (cf. property (4) of the functions $\varphi_{\ell}$), we obtain
\begin{equation}\label{goodnumberaregood}\Bigl(\sigma^{20}-\frac{C}{\D\sigma^2}\Bigl)\|f\|_{L^2(\R)}^2 \lesssim \int_{-2}^2 \sum\limits_{\ell\,:\, (\ell, \tau)\text{ is good}}\|\wh{f_{\tau, \ell}}\bigl\|_{L^2(\R)}^2d\tau,
\end{equation}
and the left hand side of this inequality can be made at least a constant multiple of $\sigma^{20}\|f\|^2_{L^2(\R)}$ by choosing $\D$ to be a suitable multiple of $\sigma^{-22}$.\\

We are now in a position to use the assumption of the PLS property for $W^{\alpha}$.  Since $E$ is $\gamma$-relatively dense, and $W^{\alpha}$ has the PLS property, there is a constant $C_{\PLS}$ depending on $W, \Lambda, \gamma$, and $\sigma$, such that for every good pair $(\tau, \ell)$ we have
\begin{equation}\label{gooduniq}\|f_{\tau, \ell}\|_{L^2(\R)}\leq C_{\PLS}\|f_{\tau, \ell}\|_{L^2(E)}.
\end{equation}
Next, since $\supp(\varphi_{\ell})\subset [-\sigma, \sigma]$ we infer that on the set $E$, $(f\cdot e_{t_{\ell}+\tau})*\varphi_{\ell} = (f\chi_{E_{\sigma}}e_{t_{\ell}+\tau})*\varphi_{\ell}$, and hence
$$\|f_{\tau, \ell}\|_{L^2(E)} = \|(f\chi_{E_{\sigma}}\cdot e_{t_{\ell}+\tau})*\varphi_{\ell}\|_{L^2(E)}\leq \|(f\chi_{E_{\sigma}}\cdot e_{t_{\ell}+\tau})*\varphi_{\ell}\|_{L^2(\R)}.
$$
Writing $\bigl\|(f\chi_{E_{\sigma}}\cdot e_{t_{\ell}+\tau})*\varphi_{\ell}\bigl\|_{L^2(\R)}=\bigl\|\wh{f\chi_{E_{\sigma}}}(\cdot-t_{\ell}-\tau)\wh{\varphi}_\ell\bigl\|_{L^2(\R)}$ (Plancherel's identity), it follows by~\ref{gooduniq} that
\begin{equation}\begin{split}\nonumber\int_{-2}^2 \sum\limits_{\ell\,:\, (\ell, \tau)\text{ is good}}\|\wh{f_{\tau, \ell}}\|^2_{L^2(\R)}d\tau&\leq C_{\PLS}\int_{-2}^2 \sum\limits_{\ell}\Bigl\|\wh{f\chi_{E_{\sigma}}}(\cdot-\tau-t_{\ell})\wh\varphi_{\ell}\Bigl\|^2_{L^2(\R)}d\tau\\&\stackrel{\text{Lemma }\ref{gensumbd}}{\lesssim} \frac{C_{\PLS}}{\sigma^2} \|\wh{f\chi_{E_{\sigma}}}\|_{L^2(\R)}^2,
\end{split}\end{equation}
 Finally, bringing ths estimate together with (\ref{goodnumberaregood}) yields 
 \begin{equation}\label{sparsebound}\|f\|^2_{L^2(\R)}\lesssim \frac{C_{\PLS} }{\sigma^{22}}\|\wh{f\chi_{E_{\sigma}}}\|^2_{L^2(\R)}\lesssim\frac{C_{\PLS} }{\sigma^{22}}  \|f\|^2_{L^2(E_{\sigma})},\end{equation}
as required.\\

Theorem \ref{sparsesupport} is proved.

\section{Proof of Theorem \ref{bourdyat}} 

Consider the weight $W(t) = \exp\bigl(\frac{t}{4\log (e+t)}\bigl)$.  Then $W$ satisfies the assumptions of Theorem \ref{sparsesupport} (for any $\alpha>0$, $W^{\alpha}$ the PLS property from Theorem \ref{PLSthm}).  Appealing to the $\varphi$-regularity of $Q$, we infer that for every $t\in \R$ and $n\geq 1$, the set $(Q-t)\cap [-e^{n+1}, -e^n]\cup [e^n, e^{n+1}]$ can be covered by at most $2 \varphi((e-1)4(n+1))\leq 2\varphi(8(n+1))$ intervals of length $\Omega_n = \log W(e^n)$ (notice that $\Omega_n\geq \tfrac{e^n}{4(n+1)}$, and $[e^n, e^{n+1}]$ has length $(e-1)e^n$).  Therefore $$\|Q\|_W\leq 2\sum_{n\geq 1} \frac{1}{n^2}\varphi(8(n+1))\leq 2\sum_{n\geq 1}\frac{1}{n^2}\varphi(16n)\leq 2\cdot 16^2\cdot C_{\varphi},$$  and we conclude that (\ref{sparseunique}) holds with $C=C(C_{\varphi},\gamma)$.

To prove the moreover statement, we need to keep track of the form of the constant $C_{\PLS}$ that appears in the proof of Theorem \ref{sparsesupport} --  see (\ref{sparsebound}).  For this we shall appeal to Proposition \ref{endpoint}. Recall that
$$\alpha = c_0\frac{\sigma}{C_{\varphi}+1},\text{ and }\, C_W = \D = C\sigma^{-22}.
$$
It clearly suffices to prove the result for large $C_{\varphi}$, and so we may assume that $\alpha$ is much smaller than one.

Since $\alpha$ is smaller than $\sigma$, from Proposition \ref{endpoint} we infer that
$$C_{\PLS}\lesssim \Bigl(\frac{C}{\gamma}\Bigl)^{\exp[\exp\{C(C_{\varphi}+1)/\sigma\}]},
$$
and therefore the constant appearing in Theorem \ref{bourdyat} may be taken of the form
$$\sigma^{-22} \Bigl(\frac{C}{\gamma}\Bigl)^{\exp[\exp\{C(C_{\varphi}+1)/\sigma\}]}\lesssim  \Bigl(\frac{C}{\gamma}\Bigl)^{\exp[\exp\{C(C_{\varphi}+1)/\sigma\}]}
$$
(where the constant $C$ has been changed to go from the first expression to the second).

\appendix

\section{Riesz Products}

The purpose of this appendix is to illustrate that one requires the interval structure in the definition of a relative dense collection of intervals to arrive at Theorems \ref{sparsesupport} or \ref{bourdyat}, and these theorems are not valid for general relatively dense sets.  We will use the well-known example of Riesz products.

Set $\T=\R/\mathbb{Z}$.  Suppose $\sigma_n\in \mathbb{N}$, $n\in \mathbb{N}$, satisfies $\sigma_{n+1}/\sigma_n>3$.  For $\|a_n\|_{\ell^{\infty}(\mathbb{N})}\leq 1$, define the Riesz product as
$$\mu =
\prod_{n=1}^{\infty}(1+a_n\cos(2\pi \sigma_n t))
$$
interpreted as a weak limit of finite products $p_N = \prod_{n=1}^{N}(1+a_n\cos(2\pi \sigma_n t))$. The sequence Fourier coefficients $\wh\mu(n) = \int_0^1 e^{-2\pi i n x}d\mu(x)$, $n\in \mathbb{Z}$, is supported in $$\Lambda = \Bigl\{\sum_{n}\eps_n \sigma_n \text{ where }\eps_n\in \{-1,0,1\}\Bigl\},$$
and, due to the lacunarity condition on $\sigma_n$, there is a unique representation of each element of $\Lambda$ (see, e.g., Chapter 2.3 of \cite{HJ}).  Recall (again see \cite{HJ}) that if
$$\sum_{n=1}^{\infty}a_n^2=+\infty,
$$
then the resulting Riesz product $\mu$ is singular with respect to Lebesgue measure.  

Put $\sigma_{n}=d^n$ for $d\in \mathbb{N}$, $d>3$.  We claim that $\Lambda$ is $\varphi$-regular, with $\varphi(t)= Ct^{\log 3/\log d}$ (see Definition \ref{regdef}).   

To verify this claim, it suffices to show that, given any point $\lambda = \sum_{n=1}\eps_n \sigma_n$, and $N, \ell\in \mathbb{N}$, the interval $[\lambda-2d^N,  \lambda+2d^N]$ can be covered by $C3^{N-\ell}$ intervals of width $d^{\ell}$. But only way to generate points in $\Lambda$ that belong to $[\lambda-2d^N, \lambda+2d^N]$ belonging to different intervals of width $d^{\ell}$ centred on $\Lambda$ is to alter the digits $\eps_n$ in the expansion $\lambda = \sum_n \eps_n \sigma_n$ with $n\in [\ell, N]$.  There are $3^{N-\ell+1}$ such digits.



Now fix $N\in \mathbb{N}$ and consider the set $$K_N = \Bigl\{x\in \T: p_N(x)\leq 2\int_{\T} p_N dx=2\;\; (\text{since }\wh{p}_N(0)=1)\Bigl\}.$$  Then $m(K_N)\geq 1/2$ (recall here that $p_n\geq 0$).  We next claim that there cannot exist an absolute constant $C>0$ such that
$$\int_{\T}|p_N|^2 dm \leq C
\int_{K_N}|p_N|^2dm \text{ for every }N\in \mathbb{N},
$$
since if there did then
$$\sup_N\int_{\T}|p_N|^2 dm\leq 4C,
$$
(recall the definition of $K_N$)
and $p_N$ would be uniformly bounded in $L^2(\T)$, contradicting the fact that $\mu$ is singular with respect to Lebesgue measure.

Now consider the measure
$$\wt\mu = \mu\psi^2 \text{ on }\R,
$$
where $|\psi|>0$ on $(-1,1)$, and $\wh{\psi}\in C_{0}^{\infty}((-1/4,1/4))$.  Let $\wt{K}_N$ be the periodization $K_N$ to $\R$, and $\wt p_N = p_N \psi \in L^2(\R)$.  Then the support of $\wt p_N$ is $\varphi$-regular with $\varphi(t) = Ct^{\log 3/\log d}$ (the convolution only spreads the support $\Lambda$ a small amount), and $K_N$ is $1/2$-relatively dense for all $N$, but
$$\sup_N \frac{\int_{\wt{K}_N}|\wt p_N|^2 dm_2}{\int_{\R} |\wt{p}_N|^2 dm_2}=0.
$$


\end{document}